\documentclass[12pt,reqno]{amsart}
\usepackage{amssymb}
\usepackage{amsmath}
\usepackage{mathrsfs}
\usepackage{amssymb, amsmath, amsfonts}

\usepackage{mathabx}

\usepackage{color}

\usepackage[hidelinks]{hyperref}

\makeatletter
\def\tank#1{\protected@xdef\@thanks{\@thanks
 \protect\footnotetext[0]{#1}}}
\def\bigfoot{

 \@footnotetext}
\makeatother

\newcommand{\ea}{\end{array}}

\allowdisplaybreaks

\newtheorem{theorem}{Theorem}[section]
\newtheorem{lemma}{Lemma}[section]
\newtheorem{proposition}[theorem]{Proposition}

\newtheorem{cor}[theorem]{Corollary}

\newtheorem{rem}{Remark}[section]

\def\beq{\begin{equation}}
\def\nneq{\end{equation}}

\def\bthm{\begin{theorem}}
\def\nthm{\end{theorem}}

\def\blem{\begin{lemma}}
\def\nlem{\end{lemma}}
\def\bprf{\begin{proof}}
\def\nprf{\end{proof}}
\def\bprop{\begin{prop}}
\def\nprop{\end{prop}}
\def\brmk{\begin{rem}}
\def\nrmk{\end{rem}}

\def\bexa{\begin{exa}}
\def\nexa{\end{exa}}
\def\bcor{\begin{cor}}
\def\ncor{\end{cor}}

\newcommand{\ee}{\mathbb{E}}

\newcommand{\rr}{\mathbb{R}}
\newcommand{\pp}{\mathbb{P}}
\newcommand{\qq}{\mathbb{Q}}

\def\FF{\mathcal F}

\def\HH{\mathcal H}

\def\SS{\mathcal S}

\title[Transportation inequalities  for stochastic heat equations]{Transportation inequalities under uniform metric  for a stochastic heat equation  driven by time-white and  space-colored noise}

\author{Shijie Shang}
 \curraddr[Shang, S.]{School of Mathematics, University of Science and Technology of China, Hefei, 230026, China.}
\email{sjshang@ustc.edu.cn}

\author{Ran Wang}
\curraddr[Wang, R.]{School of Mathematics and Statistics, Wuhan University, Wuhan, 430072, China}
\email{rwang@whu.edu.cn}

\date{}
\begin{document}
\maketitle

\noindent{\bf Abstract}
 In this paper, we prove   transportation inequalities on the space of  continuous paths with respect to the uniform metric, for the law of solution  to a stochastic heat equation  defined on $[0,T]\times [0,1]^d$. This equation is driven by the Gaussian noise, white in time and colored in space. The proof is based on a new moment inequality under the uniform metric for the stochastic convolution with respect to the time-white and space-colored noise, which is of independent interest.
\vskip0.3cm
%  Keyword is required.

\vskip0.3cm

\maketitle
%  \subjclass is required.

%%%%%%%%%%%%%%%%%%%%%%%%%%%%%%%%%%%%%%%%%%%%%%%%%%%%%%%%%%%%

\section{Introduction}

    The purpose of this paper is to study Talagrand's $T_2$-transportation inequality for the following $d$-dimensional spatial stochastic heat equation on $[0,1]^d$,
\begin{equation}\label{SPDE}
    \begin{cases}   \frac{\partial }{\partial t} u (t,x)=\Delta u(t,x)+ \sigma(u(t,x))\dot{F}(t,x)+b(u(t,x)),\ \  t\ge0, \ x\in(0,1)^d,\\
    u(t,x)=0,\ \ \ x\in \partial ([0,1]^d),\\
     u(0,x)=u_0(x),\ \ \  x\in[0,1]^d,
                        \end{cases} \end{equation}
          where $\Delta$ is the Laplacian operator on $(0,1)^d$, $\partial ([0,1]^d)$ is the boundary of $[0,1]^d$, and $u_0$ is a continuous function on $[0,1]^d$ with $u_0(x)=0$ for any $x\in \partial ([0,1]^d)$.
 Assume that the coefficients  $\sigma$ and $b$ satisfy the following conditions:
 \begin{itemize}
   \item[(C1).]  $\sigma$ and $b$ are Lipschitzian,
       i.e., there exist some constants $L_{\sigma}, L_{b}\in[0,\infty)$ such that
\begin{equation}\label{Lip}
|\sigma(v_1)-\sigma(v_2)|\le L_{\sigma}|v_1-v_2|, \ \ |b(v_1)-b(v_2)|\le L_{b}|v_1-v_2|, \ \ \ \forall v_1,v_2\in \mathbb R.
\end{equation}
   \item[(C2).]   $\sigma$ is bounded, i.e., there exists a constant $K_{\sigma}\in(0,\infty)$ such that
   \begin{equation}\label{bound}
|\sigma(v)|\le K_{\sigma},\ \ \  \forall v\in \mathbb R.
\end{equation}

 \end{itemize}

 Throughout this paper, we work on a filtered probability space $(\Omega,\mathcal{F},\{\mathcal{F}_t\}_{t\ge0},\mathbb{P})$, where the filtration $\{\mathcal{F}_t\}_{t\ge0}$ satisfies the usual conditions.
The noise $F=\{F(\varphi),\varphi:\rr^{d+1}\rightarrow \rr\}$ is an $L^2(\Omega,\FF,\pp)$-valued Gaussian process with mean zero and covariance functional given by
\begin{equation}\label{covariance1}
J(\varphi,\psi):=\int_{\rr_+}ds\int_{\rr^d}dx\int_{\rr^d}dy \varphi(s,x)f(x-y)\psi(s,y),  \ \ \varphi,\psi\in  \SS(\rr^{d+1}) ,
\end{equation}
where  $f:\rr^d\to\rr_+$ is  continuous on $\rr^d\setminus \{0\}$, and
 $\SS(\rr^{d+1})$    is  the space of all Schwartz functions on $\rr^{d+1}$, all of whose derivatives are rapidly decreasing.
 As a covariance functional of a Gaussian process, the functional  $J(\cdot, \cdot)$ should be non-negative definite, this implies that $f$ is symmetric ($f(x)=f(-x)$ for all $x\in \mathbb R^d$), and is equivalent to the existence of a non-negative tempered measure $\lambda$ on $\mathbb R^d$, whose Fourier transform is $f$. More precisely, the relationship between
 $\lambda$ and $f$ is that for all $\varphi\in \mathcal S(\mathbb R^d)$,
\begin{equation}\label{eq f lambda}
\int_{\mathbb R^d} f(x)\varphi(x)dx=\int_{\mathbb R^d}\mathcal F\varphi(\xi)\lambda(d\xi),
\end{equation}
here  $\mathcal F\varphi$ is the Fourier transform of $\varphi$, $
 \mathcal F\varphi(\xi):=\int_{\mathbb R^d}\exp(-2i\pi\xi\cdot x)\varphi(x)dx$. See \cite{Dal} for details.

In this paper,   we assume the following hypothesis:
 \begin{itemize}
   \item[$(H_{\eta}).$] There exists a constant  $\eta\in[0,1)$ satisfying that
\begin{equation}\label{H eta1}
  K_{\eta}:=\int_{\rr^d}\frac{\lambda(d\xi)}{(1+|\xi|^2)^{\eta}}<+\infty.
\end{equation}

\end{itemize}

For instance, the tempered measure $\lambda$ associated with  the function $f(x)=|x|^{-\kappa},\kappa\in(0,2\wedge d)$,   satisfies \eqref{H eta1}.

\vskip 0.3cm

As in R. Dalang \cite{Dal}, the Gaussian process $F$ can be extended to a worthy martingale measure, in the sense of J.  Walsh \cite{Walsh}, thus we can use the  Walsh integral against $F$ to
%the Walsh integral with to $F$ is well-defined. Now, we can
give the definition of mild solutions to (\ref{SPDE}) as follows.
%\begin{align}\label{eq M}
%M=\{M_t(A)=F(t, A),\ \ t\in \rr_+, A\in\BB_b(\rr^d)\},
%\end{align}
%where $\BB_b(\rr^d)$ denotes the collection of all bounded Borel measurable sets in $\rr^d$.
A continuous adapted random field $u=\{u(t,x), (t,x)\in\rr_+\times[0,1]^d\}$ is called a mild solution of equation \eqref{SPDE}, if it satisfies
\begin{align}\label{solution}
 u(t,x)=& \int_{[0,1]^d} G_{t}(x,y) u_0(y) dy + \int_0^t\int_{[0,1]^d}G_{t-s}(x,y)\sigma(u(s,y))F(ds,dy)\notag\\
 &+\int_0^tds\int_{[0,1]^d}dyG_{t-s}(x,y)b(u(s,y)), \quad \mathbb{P}-a.s., \quad x\in [0,1]^d, \quad t\geq 0 ,
 \end{align}
where $G_t(x, y)$ is the Green kernel associated with the  heat equation on $[0,1]^d$:
\begin{equation}\label{heat kernel}
    \begin{cases}   \frac{\partial}{\partial t}G_t(x,y)=\Delta_x G_t(x,y),\ \ \  t\ge0,\ x,y\in(0,1)^d,\\
    G_t(x,y)=0,\ \ \ x\in \partial ([0,1]^d),\\
     G_0(x,y)=\delta(x-y).
                        \end{cases} \end{equation}

%Denote $D_T^d:=[0,T]\times [0,1]^d$. If $d>1$, the evolution equation can not be driven by the Browinan sheet because $G$ does not belong to $L^2(D_T^d)$  and we need to work with a smoother noise.

The study of existence and uniqueness of solution  to equation \eqref{solution} on $\rr^d$ has been studied by R. Dalang in \cite{Dal}. Many other authors have also studied  $d$-dimensional spatial stochastic heat equations, see \cite{FK,MS} and references therein.
Under the assumption (C1) and $(H_\eta)$ for $\eta\in[0,1)$, D. M\'{a}rquez-Carreras and M. Sarr\`{a} \cite{MS} proved that equation \eqref{SPDE} admits a unique solution.

\vskip 0.5cm

In this paper, we shall study Talagrand's transportation inequality for solution of equation \eqref{SPDE} under the uniform metric. Let us first recall the transportation inequality.  Let $(E,d)$ be a metric space, and $\mathcal M(E)$ be the space of all probability measures on $E$.
%  equipped with $\sigma$-field $\mathcal B$ such that $d(\cdot, \cdot)$ is $\mathcal B\times \mathcal B$ measurable.
  Given    $\mu,\nu\in \mathcal M(E)$ and $p\ge1$, the Wasserstein distance is defined by
\[
 W_{p}(\mu, \nu):=\inf_{\pi}\left[\int_{E}\int_E d(x,y)^p \, \pi(dx,dy) \right]^{\frac1p},
\]
where the infimum is taken over all  the probability measures  $\pi$   on $E\times E$ with marginal distributions $\mu$ and $\nu$.  The relative entropy of $\nu$ with respect to $\mu$ is defined as
\begin{equation}\label{rate function1}
 H(\nu|\mu):=\left\{
       \begin{array}{ll}
         \int \log \frac{d\nu}{d\mu}d\nu,   & \text{if } \nu\ll \mu ;\\
        +\infty, & \hbox{\text{otherwise}.}
       \end{array}
     \right.
\end{equation}
The probability measure $\mu$ satisfies the $T_p$-transportation inequality on $(E, d)$ if there exists a constant $C>0$ such that for any probability measure $\nu$ on $E$,
\[
W_{p}(\mu,\nu)\le \sqrt{2C  H(\nu|\mu)}.
\]
As usual, we write $\mu\in {\bf T_{p}}$ for this relation. The properties ${\bf T_1}$ and ${\bf T_2}$ are particularly interesting.

It is known that Talagrand's transportation inequality  is closely  related to the concentration of measure phenomenon, the log-Sobolev  and Poincar\'e inequalities, see for instance  monographs \cite{BGL, Ledoux, M, Villani}.  Recently, the problem of  transportation inequalities to stochastic (partial) differential equations has been widely studied  and is still a very active research area from both a theoretical and an applied point of view, for example see  \cite{DP, Goz, La} and references therein.

 The work of M. Talagrand \cite{Tal} on the Gaussian measure had been  generalized by D. Feyel and A. S. \"Ustunel  \cite{FU} to the framework of the  abstract Wiener space. For  stochastic differential equations (SDEs for short), by means of Girsanov transformation and the martingale representation theorem, the ${\bf T_2}$   w.r.t. the $L^2$ and the Cameron-Martin metrics were established by  H.  Djellout {\it et al.} \cite{DGW}; the ${\bf T_2}$ w.r.t. the uniform metric was obtained by L. Wu and Z. Zhang \cite{WZ2004}.  J. Bao {\it et al.} \cite{BWY} established the ${\bf T_2}$ w.r.t. both the uniform and the $L^2$ metrics on the path space for the segment processes associated with a class of neutral functional stochastic differential equations.  B. Saussereau \cite{Sau} studied the ${\bf T_2}$ for SDEs driven by fractional Brownian motion, and  S. Riedel \cite{Rie} extended this result to SDEs driven by general Gaussian processes by using Lyons' rough paths theory.

For  stochastic partial  differential equations (SPDEs for short), L. Wu and Z. Zhang  \cite{WZ2006} studied the ${\bf T_2}$ w.r.t. the $L^2$ metric by Galerkin  approximations. B. Boufoussi and S. Hajji   \cite{BH} obtained the ${\bf T_2}$ w.r.t. the $L^2$ metric for stochastic heat equations driven by space-time white noise by Girsanov transformation,  while D. Khoshnevisan and A. Sarantsev \cite{KS} studied this problem for more general SPDEs. T. Zhang and the first named author  \cite{SZ}  established the  ${\bf T_2}$ w.r.t. the uniform metric for the stochastic heat equation driven by multiplicative space-time white noise. The above results are all forced on the SPDEs with   deterministic initial values. Recently, F.-Y. Wang and T. Zhang  \cite{WZ} studied the transportation inequalities
for SPDEs with random initial values.

 The aim of this paper is to prove that under the uniform metric, the ${\bf T_2}$ holds for stochastic heat equations  driven by multiplicative time-white and  space-colored noise. Our new contribution is the $p$th-moment inequality under the uniform metric for the stochastic convolution with respect to the  time-white and  space-colored noise, which is of independent interest.

 The rest of the paper is organized as follows. In Section 2, we  first recall some facts about the stochastic integrals with respect to
the  time-white and  space-colored noise, and establish the $p$th-moment inequality under the uniform metric for the stochastic convolution with respect to the  time-white and  space-colored noise. In Section 3, we prove the ${\bf T_2}$ for the law of the solution to equation  \eqref{SPDE}.

\vskip 1.5cm

\section{Stochastic integrals and moment estimates for the stochastic evolution}

\subsection{Stochastic integrals with respect to  time-white and  space-colored noise}

Recall \eqref{eq f lambda}. Denote by $\HH$ the Hilbert space obtained by the completion of $\SS(\mathbb R^d)$ with respect to the inner product
\begin{align}\label{190719.194202}
\langle \phi,\psi\rangle_{\HH}:=&\int_{\mathbb R^{ d}}dx\int_{\mathbb R^{ d}}dy\ \phi(x)\psi(y)f(x-y) \nonumber\\
=&\int_{\mathbb R^d}\lambda(d\xi)\FF\phi(\xi)\overline{\FF \psi}(\xi), \ \ \ \ \phi,\psi\in \SS(\mathbb R^d),
\end{align}
here $\overline z$ is the conjugate of the complex number $z$. The norm induced by $\langle\cdot,\cdot\rangle_{\HH}$ is denoted by $\|\cdot\|_{\HH}$.

%Recall that $M_t(A)=F(t, A)$ defined by \eqref{eq M}.
%Let $\GG_t$ be the completion of the $\sigma$-field generated by the random variables $\{M_s(A);0\le s\le t, A\in \BB_b(\mathbb R^d)\}$.
%Then the process $\{M_t(A), \mathcal G_t, t\in[0,T], A\in \mathbb B_b(\mathbb R^d)\}$ defines a worthy martingale in the sense of Walsh \cite{Walsh}, whose covariance measure is determined by
%$$
%\langle M(A), M(B)\rangle_t=t\int_{A}\int_{B} f(x-y)dxdy.
%$$
%For any $\phi\in \mathcal S(\rr^d)$, it can be checked that
%\begin{equation*}
%M_t(\phi)=\int_0^t\int_{\rr^d}\phi(x) M(ds,dx),
%\end{equation*}
%where the integral on the right-hand side is Walsh's stochastic integral, see \cite{Dal}.
Recall that the predictable $\sigma$-field on $\Omega\times[0,T]$ is generated by the sets $\{(s,t]\times A;A\in\FF_s, 0\le s< t\le T\}$. There are two ways to define the stochastic integral against   time-white and  space-colored noise, see \cite{DQ} for details.

On the one hand,
as in R. Dalang \cite{Dal}, the Gaussian process $F$ can be extended to a worthy martingale measure in the sense of J. Walsh \cite{Walsh}.  By using the approximation technique, for any $\HH$-valued predictable process
 $g\in L^2(\Omega\times[0,T];\HH)$,  the stochastic integral
%  of $g$ with respect to $F$
% is defined by
\begin{equation}\label{eq int1}
%\int_0^t\int_{\rr^d}g(s, x) M(ds,dx)=:
\int_0^T\int_{\rr^d}g(s, x) F(ds,dx)
\end{equation}
is well-defined, see \cite{Dal}. Furthermore, the above Walsh integral can be defined for $g\in L^2([0,T];\HH)$, $\mathbb{P}$-a.s., by using localization techniques.

On the other hand, for any orthonormal basis $\{e_k\}_{k\ge1}$ of the Hilbert space $\HH$, the family of processes
\[
\left\{B_t^k:=\int_0^t\int_{\mathbb{R}^d}e_k(y)F(ds,dy), k\ge1\right\}
\]
are a sequence of independent standard Wiener processes and the process
$
B_t:=\sum_{k\ge1}B_t^ke_k
$
is a cylindrical Brownian motion on $\HH$.
% see \cite{DQ}.
%Define the predictable $\sigma$-field in $\Omega\times[0,T]$ generated by the sets $\{(s,t]\times A;A\in\FF_s, 0\le s\le t\le T\}$.
It is well-known that (see \cite{DPZ} or \cite{DQ}) for any $\HH$-valued predictable process
 $g\in L^2(\Omega\times[0,T];\HH)$,
 %for  any predictable square-integrable process $g$  valued in $\HH$,
we can define the stochastic integral with respect to the cylindrical Wiener process $B$ as follows:					
\begin{equation}\label{eq int2}
\int_0^T g(s)dB_s:=\sum_{k\ge1}\int_0^T \langle g(s),e_k\rangle_{\HH} dB_s^k.
\end{equation}
  Note that the above series converges in $L^2(\Omega, \FF,\pp)$ and the sum does not depend on
the selected orthonormal basis. Moreover, each summand, in the above series, is a classical It\^o
integral with respect to a standard Brownian motion.

By \cite[Proposition 2.6]{DQ}, we know that for any $\HH$-valued predictable process
 $g\in L^2(\Omega\times[0,T];\HH)$,
\begin{align}\label{190724.213830}
	\int_0^T \int_{\rr^d}g(s, x) F(ds,dx) = \int_0^T g(s)dB_s.
\end{align}

\vskip 1.0cm

\subsection{Moment estimates for the stochastic convolution under the uniform metric}

In this part, we will establish some moment estimates for the stochastic convolution driven by the time-white and  space-colored noise. This part is inspired by \cite{SZ} in the case of the space-time white noise.

For any random variable $\xi\in L^p(\Omega)$, let $\|\xi\|_{L^p(\Omega)}:=\left(\mathbb E |\xi|^p\right)^{\frac{1}{p}}$.

\begin{proposition}\label{estimates 0}
  Let $\{\sigma(s,y), (s,y)\in[0,T]\times [0,1]^d\}$ be a predictable random field
%   such that  $\sup_{(s,y)\in[0,T]\times [0,1]^d}\|\sigma(s,y)\|_{L^p(\Omega)}<\infty$ for some even integer $p\ge2$,
    and $\{p_s(x,y), (s,x,y)\in[0,T]\times [0,1]^d\times [0,1]^d\}$ be a deterministic function such that the following stochastic integral is well-defined.
    Then for any $t\in[0,T], x\in [0,1]^d$ and $p\geq 2$,
  \begin{align}\label{101.01}
    & \mathbb E\left[ \left|\int_0^t\int_{[0,1]^d} p_{t-s}(x,y)\sigma(s,y)F(ds,dy) \right|^p\right]\nonumber\\
    \leq & (4p)^{\frac{p}{2}}\left(\int_0^t \left\|p_{t-s}(x,\cdot) \|\sigma(s, \cdot)\|_{L^p(\Omega)}\right\|_{\HH}^2ds\right)^{\frac p 2}.
  \end{align}
\end{proposition}

\begin{proof}
By Burkholder's inequality (see \cite[Proposition 4.4]{K} and \cite[Page 14]{Dal}), we have
\begin{align}
&\mathbb E  \left|\int_0^t\int_{[0,1]^d} p_{t-s}(x,y)\sigma(s,y)F(ds,dy) \right|^p \nonumber\\
\le & (4p)^{\frac{p}{2}}\mathbb E \left|\int_0^tds \int_{[0,1]^{d}}dy \int_{[0,1]^{d}}dz \  p_{t-s}(x, y)p_{t-s}(x, z)f(y-z)\sigma(s,y)\sigma(s,z)\right|^{p/2} .
\end{align}
Taking $\frac{2}{p}$-th power on both sides of the above inequality and using H\"older's inequality, we get
\begin{align}
& \left\Vert \int_0^t\int_{[0,1]^d} p_{t-s}(x,y)\sigma(s,y)F(ds,dy) \right\Vert_{L^p(\Omega)}^2 \nonumber\\
\le & 4p \left\Vert \int_0^tds \int_{[0,1]^{d}}dy \int_{[0,1]^{d}}dz\  p_{t-s}(x, y)p_{t-s}(x, z)f(y-z)\sigma(s,y)\sigma(s,z)\right\Vert_{L^{\frac{p}{2}}(\Omega)} \nonumber\\
\le & 4p \int_0^tds \int_{[0,1]^{d}}dy \int_{[0,1]^{d}}dz\  p_{t-s}(x, y)p_{t-s}(x, z)f(y-z) \left\Vert \sigma(s,y)\sigma(s,z)\right\Vert_{L^{\frac{p}{2}}(\Omega)} \nonumber\\
\le &  4p \int_0^tds \int_{[0,1]^{d}}dy \int_{[0,1]^{d}}dz\  p_{t-s}(x, y)p_{t-s}(x, z)f(y-z) \left\Vert \sigma(s,y)\right\Vert_{L^p(\Omega)}\left\Vert\sigma(s,z)\right\Vert_{L^{p}(\Omega)} \nonumber\\
= & 4p \int_0^t \left\|p_{t-s}(x,\cdot) \|\sigma(s, \cdot)\|_{L^p(\Omega)}\right\|_{\HH}^2ds .
\end{align}
Then we take $\frac{p}{2}$-th power on both sides of the above inequality to obtain (\ref{101.01}).
The proof is complete.
\end{proof}

\begin{proposition}\label{estimates 1}
Suppose hypothesis $(H_{\eta})$ holds. Let $\{\sigma(s,y), (s,y)\in\mathbb{R}_+\times [0,1]^d\}$ be a random field such that the stochastic integral against the time-white and space-colored noise $F$ in (\ref{101.1}) is well-defined. Then for any $T>0$ and  $p>\frac{4+d}{1-\eta}$, there exists a constant $C_{T,p,\eta}>0$ such that
  \begin{align}\label{101.1}
    & \mathbb E\left[\sup_{(t,x)\in[0,T]\times[0,1]^d}\left|\int_0^t\int_{[0,1]^d} G_{t-s}(x,y)\sigma(s,y)F(ds,dy) \right|^p\right] \nonumber\\
    \leq & C_{T,p,\eta} \int_0^T  \sup_{ y\in [0,1]^d}\mathbb E\left|\sigma(s,y)\right|^p\,ds .
  \end{align}

\end{proposition}

    \begin{proof} The proof is based on the factorization method, which is  inspired by   \cite{SZ}. We shall give its proof for the completeness.
Obviously, we can assume that the right hand side of (\ref{101.1}) is finite.
 Choose $\alpha$ such that  $\frac{d+2}{2p}<\alpha<\frac{1}{2}-\frac{1}{p}-\frac{\eta}{2}$. This is possible because $p>\frac{4+d}{1-\eta}$.

 Let
\begin{align}
  (J_{\alpha}\sigma)(s,y):&= \int_0^s\int_{[0,1]^d}(s-r)^{-\alpha}G_{s-r}(y,z)\sigma(r,z)F(dr,dz), \\
  (J^{\alpha-1}f)(t,x):&= \frac{\sin(\pi\alpha)}{\pi}\int_0^t\int_{[0,1]^d}(t-s)^{\alpha-1}G_{t-s}(x,y)f(s,y)dsdy.
\end{align}
  By the stochastic  Fubini theorem,  for any $(t,x)\in\mathbb{R}_+\times[0,1]^d$,
\begin{align}\label{103.1}
  \int_0^t\int_{[0,1]^d}G_{t-s}(x,y)\sigma(s,y)F(ds,dy)=J^{\alpha-1}(J_{\alpha}\sigma)(t,x), \quad \mathbb P\text{-a.s.}.
\end{align}

\noindent Therefore
\begin{align}\label{103.11}
  & \sup_{(t,x)\in[0,T]\times[0,1]^d}\left|\int_0^t\int_{[0,1]^d}G_{t-s}(x,y)\sigma(s,y)F(ds,dy)\right| \nonumber\\
  =& \sup_{(t,x)\in[0,T]\times[0,1]^d}\left|J^{\alpha-1}(J_{\alpha}\sigma)(t,x)\right|, \quad \mathbb P\text{-a.s.}.
\end{align}
%By (\ref{190719.1930.11}) in Appendix,  we know that
%\begin{align}
%\label{103.2}  \int_{[0,1]^d} G_t(x,y)^2\,dy =& \sup_{y\in[0,1]^d}G_t(x,y)\times \int_{[0,1]^d} G_t(x,y)\,dy \leq C_2 t^{-\frac{d}{2}},
%\end{align}
%where
%\begin{align}\label{C_2}
%C_2 = (4\pi)^{-\frac{d}{2}} .
%\end{align}
By H\"{o}ler's inequality, (\ref{103.11}) and (\ref{103.2}) in Appendix, we have
{\allowdisplaybreaks\begin{align}\label{104.1}
  & \mathbb E\sup_{(t,x)\in[0,T]\times[0,1]^d}\left|\int_0^t\int_{[0,1]^d}G_{t-s}(x,y)\sigma(s,y)F(ds,dy)\right|^p \nonumber\\
  =& \mathbb E\sup_{(t,x)\in[0,T]\times[0,1]^d}\left|\frac{\sin(\pi\alpha)}{\pi} \int_0^t\int_{[0,1]^d} (t-s)^{\alpha-1} G_{t-s}(x,y)J_{\alpha}\sigma(s,y)\,ds dy \right|^p \nonumber\\
  \leq & \left|\frac{\sin(\pi\alpha)}{\pi}\right|^p \mathbb E\sup_{(t,x)\in[0,T]\times[0,1]^d}\bigg\{\int_0^t (t-s)^{\alpha-1}
\times\left(\int_{[0,1]^d} G_{t-s}(x,y)|J_{\alpha}\sigma(s,y)|\,dy\right)ds\bigg\}^p \nonumber\\
  \leq & \left|\frac{\sin(\pi\alpha)}{\pi}\right|^p \mathbb E\sup_{(t,x)\in[0,T]\times[0,1]^d}\Bigg\{\int_0^t (t-s)^{\alpha-1}
\times\left(\int_{[0,1]^d} G_{t-s}(x,y)|J_{\alpha}\sigma(s,y)|^{\frac{p}{2}}\, dy \right)^{\frac{2}{p}}ds\Bigg\}^p \nonumber\\
  \leq & \left|\frac{\sin(\pi\alpha)}{\pi}\right|^p \mathbb E\sup_{(t,x)\in[0,T]\times[0,1]^d}\Bigg\{\int_0^t (t-s)^{\alpha-1}  \nonumber\\
  &~~~~~~~~~~~~~~~~~~~~~~~~\times\left(\int_{[0,1]^d} G_{t-s}(x,y)^2\,dy\right)^{\frac{1}{p}}\left(\int_{[0,1]^d}|J_{\alpha}\sigma(s,y)|^p\,dy\right)^{\frac{1}{p}}\,ds\Bigg\}^p \nonumber\\
  \leq & \left|\frac{\sin(\pi\alpha)}{\pi}\right|^p (4\pi)^{-\frac d2}\mathbb E\sup_{t\in[0,T]}\left\{\int_0^t (t-s)^{\alpha-1-\frac{d}{2p}}\left(\int_{[0,1]^d}|J_{\alpha}\sigma(s,y)|^p\,dy\right)^{\frac{1}{p}}\,ds\right\}^p \nonumber\\
  \leq & \left|\frac{\sin(\pi\alpha)}{\pi}\right|^p (4\pi)^{-\frac d2} \mathbb E\sup_{t\in[0,T]}\Bigg[\left(\int_0^t (t-s)^{(\alpha-1-\frac{d}{2p})\frac{p}{p-1}}\,ds\right)^{p-1} \times\left(\int_0^t\int_{[0,1]^d}|J_{\alpha}\sigma(s,y)|^p\,dyds\right)\Bigg] \nonumber\\
  \leq & \left|\frac{\sin(\pi\alpha)}{\pi}\right|^p (4\pi)^{-\frac d2} \times\left(\int_0^T s^{(\alpha-1-\frac{d}{2p})\frac{p}{p-1}}\,ds\right)^{p-1}\times\int_0^T\int_{[0,1]^d}\mathbb E|J_{\alpha}\sigma(s,y)|^p\,dyds \nonumber\\
  \leq & C_{T,p,\alpha}^{\prime}\sup_{(s,y)\in[0,T]\times[0,1]^d} \mathbb E\left|\int_0^s\int_{[0,1]^d} (s-r)^{-\alpha}G_{s-r}(y,z)\sigma(r,z)F(dr,dz)\right|^p ,
\end{align}}
where we have used the condition $\alpha >\frac{ d}{ 2p}+\frac{1}{p}$, so that
\begin{align}\label{C_{T,p} prime}
  C_{T,p,\alpha}^{\prime}= & \left|\frac{\sin(\pi\alpha)}{\pi}\right|^p  (4\pi)^{-\frac d2} \times\left(\int_0^T s^{(\alpha-1-\frac{d}{2p})\frac{p}{p-1}}\,ds\right)^{p-1}\times T \nonumber\\
  =& \left|\frac{\sin(\pi\alpha)}{\pi}\right|^p {(4\pi)}^{-\frac d2} \left(\frac{p-1}{\alpha p-1-\frac{d}{2}}\right)^{p-1} T^{\alpha p-\frac{d}{2}} < \infty.
\end{align}

% Applying the BDG inequality \eqref{101.01}, \eqref{103.2} and H\"older inequality, we have
\noindent By H\"older's inequality,  \eqref{101.01} and Lemma \ref{Lem t} in Appendix, we have
\begin{align}\label{104.2}
% &  \left|\int_0^s\int_{[0,1]^d} (s-r)^{-\alpha}G_{s-r}(y,z)\sigma(r,z)F(dr,dz)\right|^2_{L^p(\Omega)} \nonumber\\
% \leq & C_3 \int_0^s(s-r)^{-2\alpha}dr\int_{[0,1]^d}dz_1\int_{[0,1]^d}dz_2G_{s-r}(y,z_1)G_{s-r}(y,z_2)  f(z_1-z_2) \nonumber\\
% &\ \ \ \ \ \ \times \left\Vert\sigma(r,z_1)\right\Vert_{L^{p}(\Omega)}  \left\Vert\sigma(r,z_2)\right\Vert_{L^{p}(\Omega)}  \nonumber\\
% \leq & C_3\int_0^s(s-r)^{-2\alpha}   \|G_{s-r}\|_{\HH}^2\cdot\sup_{z\in[0,1]^d}\left\Vert\sigma(r,z)\right\Vert_{L^p(\Omega)}^2\,dr  \nonumber\\
% \leq & C_3\left(\int_0^s(s-r)^{-\frac{2\alpha p}{p-2}}   \|G_{s-r}\|_{\HH}^2  dr\right)^{\frac{p-2}{p}}\times \left(\int_0^s \|G_{s-r}\|_{\HH}^2\cdot \sup_{z\in[0,1]^d}   \left\Vert   \sigma(r,z)\right\Vert_{L^p(\Omega)}^p\,dr\right)^{\frac{2}{p}}   \nonumber\\
%     \leq & C''_{T, \alpha, p} \left(\int_0^s \|G_{s-r}\|_{\HH}^2\cdot \sup_{z\in[0,1]^d}   \left\Vert   \sigma(r,z)\right\Vert_{L^p(\Omega)}^p\,dr\right)^{\frac{2}{p}},
 &  \mathbb{E}\left|\int_0^s\int_{[0,1]^d} (s-r)^{-\alpha}G_{s-r}(y,z)\sigma(r,z)F(dr,dz)\right|^p \nonumber\\
 \leq & (4p)^{\frac{p}{2}} \left(\int_0^s(s-r)^{-2\alpha} \cdot \|G_{s-r}\|_{\HH}^2  \cdot\sup_{z\in[0,1]^d}\left\Vert\sigma(r,z)\right\Vert_{L^p(\Omega)}^2\,dr \right)^{\frac{p}{2}} \nonumber\\
 \leq & (4p)^{\frac{p}{2}}
 \left(\int_0^s \left[(s-r)^{-2\alpha}  \|G_{s-r}\|_{\HH}^2\right ]^{\frac{p}{p-2}}     dr\right)^{\frac{p-2}{2}}
  \times \left(\int_0^s \sup_{z\in[0,1]^d}   \left\Vert   \sigma(r,z)\right\Vert_{L^p(\Omega)}^p\,dr\right)  \nonumber\\
 \leq & (4p)^{\frac{p}{2}}
 \left(\int_0^s \left[(s-r)^{-2\alpha} \times K_{\eta} \left(1\vee \left(\frac{\eta}{8\pi^2}\right)^{\eta} (s-r)^{-\eta}\right)\right]^{\frac{p}{p-2}}     dr\right)^{\frac{p-2}{2}} \nonumber\\
 &\times \left(\int_0^s \sup_{z\in[0,1]^d}   \left\Vert   \sigma(r,z)\right\Vert_{L^p(\Omega)}^p\,dr\right)   \nonumber\\
\leq & C''_{T,  p, \alpha, \eta} \int_0^s  \sup_{z\in[0,1]^d}   \mathbb{E}\left|   \sigma(r,z)\right|^p\,dr, \end{align}
where we have used  the condition $\alpha<\frac{1}{2}-\frac{1}{p}-\frac{\eta}{2}$, so that
 \begin{align}\label{190719.1723}
 &C''_{T,  p, \alpha, \eta}\nonumber\\
 = & (4p K_{\eta})^{\frac{p}{2}}\sup_{s\in[0,T]}\left(\int_0^s(s-r)^{-\frac{2\alpha p}{p-2}} \left[1\vee  \left(\frac{\eta}{8\pi^2}\right)^{\eta}(s-r)^{ -\frac{\eta p}{p-2}}\right]  dr\right)^{\frac{p-2}{2}} \nonumber\\
 \leq & (4p K_{\eta})^{\frac{p}{2}}
 \left(\frac{p-2}{p-2-2\alpha p}  T^{\frac{p-2-2\alpha p}{p-2}}+ \left(\frac{\eta}{8\pi^2}\right)^{\eta}\frac{p-2}{p-2-2\alpha p- \eta p} T^{\frac{p-2-2\alpha p-\eta p}{p-2}} \right)^{\frac{p-2}{2}} \nonumber\\
 \le & \frac{1}{4} (8pK_{\eta})^{\frac{p}{2}} \left[\left(\frac{p-2}{p-2-2\alpha p}\right)^{\frac{p-2}{2}} T^{\frac{p}{2}-1-\alpha p} + \left(\frac{\eta}{8\pi^2}\right)^{\frac{\eta(p-2)}{2}} \left(\frac{p-2}{p-2-2\alpha p -\eta p}\right)^{\frac{p-2}{2}} T^{\frac{p}{2}-1-\alpha p -\frac{\eta p}{2}}\right] \nonumber\\
 < & \infty.
 \end{align}

\noindent Combining (\ref{104.1}) with (\ref{104.2}), we obtain
\begin{align}
  & \mathbb E\sup_{(t,x)\in[0,T]\times[0,1]^d}\left|\int_0^t\int_{[0,1]^d} G_{t-s}(x,y)\sigma(s,y)F(ds,dy)\right|^p \nonumber\\
  \leq & C_{T,p,\eta} \int_0^T  \sup_{z\in [0,1]^d}\mathbb E\left|\sigma(s,z)\right|^p\,ds ,
\end{align}
where
\begin{align}\label{C_{T,p}}
  C_{T,p,\eta}= \min_{\frac{d+2}{2p}<\alpha<\frac{1}{2}-\frac{1}{p}-\frac{\eta}{2}}C^{\prime}_{T,p,\alpha}\times C^{\prime\prime}_{T,p,\alpha, \eta} .
\end{align}
In view of (\ref{C_{T,p} prime}) and (\ref{190719.1723}), a straightforward calculation leads to
\begin{align}\label{eq c t p}
  C_{T,p, \eta} < & p^{\frac{p}{2} } \times \frac{1}{4}\left(\frac{1}{4\pi}\right)^{\frac{d}{2}}\left(\frac{\sqrt{8K_{\eta}}}{\pi}\right)^p \nonumber \\
   & \times \Bigg[\max\left\{\left(\frac{3p-4}{p-4-d}\right)^{\frac{3p}{2}-2},  \left( \frac{2(p-1)}{(1-\eta)p-4-d}\right)^{p-1} \left(\frac{p-2}{p\eta}\right)^{\frac p2-1} \right\}  T^{\frac{p}{2}-1-\frac{d}{2}}\nonumber\\
    &\ \ \ +\left(\frac{\eta}{8\pi^2}\right)^{\frac{\eta(p-2)}{2}} \left(\frac{3p-4}{(1-\eta)p-4-d}\right)^{\frac{3p}{2}-2} T^{\frac{(1-\eta)p}{2}-1-\frac{d}{2}}\Bigg] .
\end{align}
This completes the proof of the estimate (\ref{101.1}).
\end{proof}

Based on Proposition \ref{estimates 1}, we can obtain   the following estimate by using the argument in the proof of  \cite[Proposition 3.4 ]{SZ}. The proof is omitted here.

\begin{proposition}\label{estimates 2}
Suppose hypothesis $(H_{\eta})$ holds.
Let $\{\sigma(s,y), (s,y)\in\mathbb{R}_+\times [0,1]^d\}$ be as in Proposition \ref{estimates 1}.
Then for any $T>0$, $\varepsilon>0$ and $0 < p\leq \frac{4+d}{1-\eta}$, there exists a constant $C_{T,p,\eta,\varepsilon}$ such that
  \begin{align}\label{101.2}
    & \mathbb{E}\left[\sup_{(t,x)\in[0,T]\times[0,1]^d}\left|\int_0^t\int_{[0,1]^d} G_{t-s}(x,y)\sigma(s,y)F(ds,dy) \right|^p\right] \nonumber\\
    \leq & \varepsilon  \mathbb{E}\left[\sup_{(s,y)\in[0,T]\times[0,1]^d}\left|\sigma(s,y) \right|^p\right] + C_{T,p,\eta,\varepsilon} \mathbb{E}\int_0^T \sup_{y\in[0,1]^d}\left|\sigma(s,y)\right|^p\,ds,
  \end{align}
where
 \begin{align}\label{101.3}
  C_{T,p,\eta,\varepsilon}=\inf_{q>\frac{4+d}{1-\eta} }\left\{\left(1+ \frac{q C_{T, p,\eta}}{q-p}  \right) pq^{-\frac{q}{p}}\left(q-p+q C_{T,p, \eta}
  \right)^{\frac{q}{p}-1} \varepsilon^{\frac{q}{p}-1}  \right\},
\end{align}
and the constant $C_{T, p,\eta}$ is bounded by the right hand side of \eqref{eq c t p}.

\end{proposition}

\vskip 1.5cm

\section{Transportation inequality}

\subsection{The main results}
Let  $E:=\mathcal C_0([0,1]^d;\rr)$ be the space of all continuous functions $f$ from $[0,1]^d$ to $\rr$ satisfying $f(x)=0$ for all $x\in \partial ([0,1]^d)$, endowed with the uniform metric
$$
d_{E, \infty}(f_1, f_2):=\sup_{x\in  [0,1]^d} |f_1(x)-f_2(x)|, \ \ \ f_1, f_2\in E,
$$
 and  let $E_T:= \mathcal C([0,T]; E)$
 be the space of all continuous functions from $[0,T]$ to $E$, endowed with the uniform metric
$$
d_{E_T, \infty}(u_1, u_2):=\sup_{(t,x)\in \times [0,T]\times[0,1]^d} |u_1(t, x)-u_2(t, x)|, \ \ \ u_1, u_2\in  E_T.
$$

Let $\{u(t,x), (t,x)\in[0,T]\times [0,1]^d\}$ be the unique solution of the equation \eqref{SPDE}.
For any  $\mu\in \mathcal M(E)$, let $P^{\mu}$ be the distribution of the solution  $\{u(t,x), (t,x)\in[0,T]\times [0,1]^d\}$ on $E_T$ such that the law of $u_0$ is $\mu$.  Particularly, if $\mu=\delta_{u_0}$ for some $u_0\in E$, we write $P^{u_0}:=P^{\delta_{\mu_0}}$ for short.

% In this paper, we establish the following result:
Here are the main results of this section.
 \bthm\label{thm transport}  Under  (C1), (C2) and $(H_{\eta})$, for any deterministic initial value $u_0\in E$.  the law $P^{u_0}$ satisfies the $T_2$-transportation inequality on the space $E_T$   with respect to the uniform metric.
 \nthm

Applying   Proposition \ref{estimates 2} and  Theorem \ref{thm transport}   and   using the same  approach  in the proof of \cite[Theorem 3.1]{WZ},  we can get the following  transportation inequality   for the stochastic
heat equation  with random initial values,  whose proof is omitted here.
\bcor Under  (C1), (C2) and $(H_{\eta})$, and $\mu\in \mathcal M(E)$. Then
\begin{equation}
W_2^2(Q, P^{\mu})\le 2CH(Q|P^{\mu}), \ \ \forall Q\in \mathcal M(E_T)
\end{equation}
holds for some constant $C>0$ if and only if
\begin{equation}
W_2^2(\nu, \mu)\le 2c H(\nu|\mu), \ \ \forall  \nu\in \mathcal M(E)
\end{equation}
holds for some constant $c>0$.
\ncor
\vskip0.3cm
\subsection{The proof of Theorem \ref{thm transport}}

We will apply the Girsanov theorem to prove Theorem \ref{thm transport}. To do this, we need the following lemma describing  all probability measures which are absolutely continuous with respect to $P^{u_0}$. It is analogous to \cite[Theorem 5.6]{DGW} in the setting of finite-dimensional Brownian motion  and  \cite[Lemma 3.1]{KS} in the setting of space-time white noise. For the completeness, we give its proof here.

\blem\label{lem Girsanov}
Let $Q$ be a probability measure on $E_T$ such that $Q\ll P^{u_0}$.
 Define a new probability measure $\mathbb Q$ on the probability space $(\Omega,\mathcal{F},\mathbb{P})$ by
  \begin{equation}\label{eq Q}
  d\mathbb Q:=\frac{dQ}{d P^{u_0}}(u)d\mathbb P.
\end{equation}
Then there exists an $\HH$-valued predictable process  $h=\{h(s), s\in[0,T] \}$ such that
\begin{align}
	\int_0^T \|h(s)\|_{\HH}^2ds<\infty, \quad \mathbb Q-a.s.,
\end{align}
and the process
\begin{align}\label{190725.224510}
	\widetilde B_t := B_t -\int_0^t h(s)ds, \quad t\in[0,T],
\end{align}
is a cylindrical Wiener process on $\HH$ under  $\mathbb Q$. Furthermore,
\beq\label{eq entropy}
H(Q|P^{u_0})=\frac{1}{2} \ee^{\mathbb Q} \int_0^T\|h(s)\|_{\HH}^2ds ,
\nneq
where $\ee^{\mathbb Q}$ denotes the expectation under the probability measure $\mathbb{Q}$.

\nlem

\begin{proof}
The proof is adapted  from  \cite[Theorem 5.6]{DGW}.
Let
\[
M_t:=\frac{d\mathbb Q}{d\mathbb P}\Big|_{\mathcal F_t}, \quad t\in[0,T].
\]
Then $(M_t)_{t\in[0,T]}$
is a nonnegative $\mathbb P$-martingale.
Let $\tau:=\inf\{t\ge0; M(t)=0\}\wedge T$ with the convention $\inf\emptyset :=\infty$. Then $\mathbb Q(\tau=T)=1$, and  the martingale $M$ can be represented as the stochastic exponential of another continuous local martingale $L$:
\begin{align}\label{eq exp}
M(t)=\exp\left(L(t)-\frac{1}{2}[L]_t\right), \ \ \ t<\tau,
\end{align}
where $L(t)=\int_0^t\frac{dM(s)}{M(s)}$ for $t<\tau$. By the martingale representation theorem (e.g., \cite[Theorem 2.3]{BD}), we know that there is a predictable process $h=\{h(t);t\in[0,\tau)\}$ valued in $\HH$ such that
\[
\int_0^t\|h(s)\|_{\HH}^2ds <+\infty, \quad t<\tau, \quad \mathbb P-a.s.,
\]
and
\[
L(t)=\int_0^t h(s)dB_s=\sum_{k=1}^{\infty}\int_0^t\langle h(s),e_k\rangle_{\HH}dB_s^k, \quad t<\tau,
\]
where the above integral is defined as in \eqref{eq int2}. By the Girsanov  theorem (e.g., \cite[Theorem 2.2]{BD} or \cite[Theorem 10.14]{DPZ}), we know that
$$
	\widetilde B_t := B_t -\int_0^t h(s)ds, \quad t\in[0,T]
$$
is a cylindrical Wiener process on $\HH$ under  $\mathbb Q$.

Let $\tau_n:=\inf\left\{t\in[0,\tau); [L]_t=n\right\}\wedge\tau$ with the same convention that $\inf\emptyset:=\infty$. Then $\tau_n\uparrow \tau$, $\mathbb P$-a.s., and by the martingale convergence theorem, we have
\begin{align*}
H(Q|P^{u_0})=&\mathbb E^{\mathbb P}M_T\log M_T=
%\mathbb E^{\mathbb P}M_{\tau}\log M_{\tau} =
\lim_{n\rightarrow\infty} \mathbb E^{\mathbb P}M_{T\wedge\tau_n}\log M_{T\wedge\tau_n}=
%\mathbb{E}^{\mathbb Q} \log M_T =
\lim_{n\rightarrow\infty} \mathbb E^{\mathbb Q}\log M_{T\wedge\tau_n}\\
=&\lim_{n\rightarrow\infty} \mathbb E^{\mathbb Q} \left(L(T\wedge\tau_n)-\frac12[L]_{T\wedge \tau_n}\right).
\end{align*}
By Girsanov  formula, $\{L(t\wedge\tau_n)-[L]_{t\wedge \tau_n}, t\in[0,T]\}$
is a $\mathbb Q$-local martingale, then a true martingale since its quadratic variation process under $\mathbb Q$, being again $\{[L]_{t\wedge \tau_n}, t\in[0,T]\}$, is bounded by $n$. Consequently, $\mathbb E^{\mathbb Q}(L(T\wedge\tau_n)-[L]_{T\wedge \tau_n})=0$. Substituting it into the proceeding equality and by the monotone convergence, we have
\[
H(Q|P^{u_0})=\frac{1}2\lim_{n\rightarrow\infty} \mathbb E^{\mathbb Q} ([L]_{T\wedge \tau_n})=\frac12\mathbb E^{\mathbb Q}[L]_T=\frac{1}{2}\mathbb E^{\mathbb Q}\int_0^T\|h(s)\|_{\HH}^2ds.
\]
The proof is complete.
\end{proof}

%\bprf
%  It is enough to prove the result for any probability measure $\nu$ on $\mathcal C([0,T]\times [0,1]^d;\mathbb R)$
% such that $\nu\ll \mu$ and $ H(\nu|\mu)<\infty$. Let $(\Omega, \mathcal F, \mathbb P)$ be a complete probability space on which $(B_t)=(B_t^k)_{k\ge1}$ is a cylindrical Wiener process on $\HH$ and let $\mathcal  F_t=\mathcal F_t^B=\sigma(B_s, s\le t)^{\mathbb P}$ (completion by $\mathbb P$). Let $u(t,x)$ be the unique solution of \eqref{SPDE}. Then the law of $u(t, \cdot)$ is $\mu$.
%  Define a new probability measure $\mathbb Q$ on the filtered probability space by
%  \begin{equation}\label{eq Q}
%  d\mathbb Q:=\frac{d\nu}{d\mu}(u)d\mathbb P.
%\end{equation}
%Let $M_t:=\frac{d\mathbb Q}{d\mathbb P}\big|_{\mathcal F_t}, \ \ t\in[0,T]$. Then $(M_t)_{t\in[0,T]}$ is a $\mathbb P$-martingale. Let $h(t,x)$ be the corresponding random field appeared in Lemma \ref{lem Girsanov}. Then by  Lemma \ref{lem Girsanov},     the  process
%$$
%\widetilde B^k_t=B_t^k-\int_0^t\langle h(s),e_k\rangle_{\HH}ds,\ \ \ k\ge1,
%$$
%are a sequence of independent standard Wiener processes under $\mathbb Q$.

\vskip 0.5cm

\bprf[Proof of Theorem \ref{thm transport}]
For any $Q\ll P^{u_0}$ such that $ H(Q|P^{u_0})<\infty$, let $\mathbb Q$ be defined as (\ref{eq Q}) and $h$ be the corresponding process appeared in Lemma \ref{lem Girsanov}. Then it is easy to see that the solution of equation \eqref{SPDE} satisfies the following equation:
\begin{align}\label{eq u}
 u(t,x)=& \int_{[0,1]^d} G_{t}(x,y) u_0(y)dy +\sum_{k\ge1}\int_0^t \langle  G_{t-s}(x, \cdot)\sigma(u(s,\cdot)),e_k\rangle_{\HH} d\widetilde B_s^k\notag\\
 &+\int_0^t\int_{[0,1]^d} G_{t-s}(x,y)b(u(s,y))dyds\notag\\
 &+\sum_{k\ge 1}\int_0^t\left\langle G_{t-s}(x, \cdot) \sigma(u(s,\cdot)), e_k\right\rangle_{\HH} \cdot\left\langle h(s), e_k\right\rangle_{\HH}
ds.
                  \end{align}
Consider the solution of the following equation:
\begin{align}\label{eq v}
 v(t,x)= & \int_{[0,1]^d} G_{t}(x,y) u_0(y)dy +\sum_{k\ge1}\int_0^t \langle  G_{t-s}(x, \cdot)\sigma(v(s,\cdot)),e_k\rangle_{\HH} d\widetilde B_s^k\notag\\
 &+\int_0^t\int_{[0,1]^d} G_{t-s}(x,y)b(v(s,y))dyds.
                  \end{align}
 By Lemma \ref{lem Girsanov}, \eqref{eq int2}, and (\ref{190724.213830}), it follows that under probability measure $\mathbb Q$, the law of $(v,u)$ forms a coupling of $(\mu,\nu)$. Therefore, by the definition of the Wasserstein distance, we have
\beq\label{eq uv}
W_{2}^2(Q, P^{u_0}) \le \ee^{\mathbb Q}\left[\sup_{(t,x)\in [0,T]\times [0,1]^d} |u(t,x)-v(t,x)|^2  \right].
\nneq
In view of \eqref{eq entropy} and \eqref{eq uv},  to prove the $T_2$-transportation inequality, it is sufficient to show that
\beq\label{eq tar}
\ee^{\qq}\left[\sup_{(t,x)\in [0,T]\times [0,1]^d} |u(t,x)-v(t,x)|^2 \right]\le C \ee^{\qq}\int_0^T\|h(s)\|^2_{\HH}ds.
\nneq
for some constant $C$ independent of $v$. From   \eqref{eq u} and \eqref{eq v}, we have
\begin{align}\label{eq u-v}
 u(t,x)-v(t,x)=&\sum_{k\ge 1}\int_0^t\left\langle G_{t-s}(x-\cdot)[\sigma(u(s,\cdot))-\sigma(v(s,\cdot))] , e_k\right\rangle_{\HH}
d\widetilde B_s^k\notag\\
 &+\int_0^t\int_{[0,1]^d} G_{t-s}(x,y)[b(u(s,y))-b(v(s,y))]dyds\notag\\
 &+ \sum_{k\ge 1} \int_0^t\left\langle G_{t-s}(x,\cdot)\sigma( u(s,\cdot)), e_k\right\rangle_{\HH}\cdot\left\langle h(s), e_k\right\rangle_{\HH}
ds .
   \end{align}
 Thus,
 \begin{align}\label{eq T}
 |u(t,x)-v(t,x)|^2\le &3\left|\sum_{k\ge 1}\int_0^t\left\langle G_{t-s}(x,\cdot)[\sigma(u(s,\cdot))-\sigma(v(s,\cdot))], e_k\right\rangle_{\HH}
d\widetilde B_s^k\right|^2  \notag\\
 &+3\left|\int_0^t\int_{[0,1]^d} G_{t-s}(x,y)[b(u(s,y))-b(v(s,y))]dyds\right|^2\notag\\
 &+3\left|\sum_{k\ge 1} \int_0^t\left\langle G_{t-s}(x,\cdot)\sigma( u(s,\cdot)), e_k\right\rangle_{\HH}\cdot\left\langle h(s), e_k\right\rangle_{\HH}ds\right|^2\notag\\
 =:& 3[ |I_1(t,x)|^2 + |I_1(t,x)|^2 + |I_3(t,x)|^2].
   \end{align}

\noindent
By Proposition \ref{estimates 2}, we obtain that for any $\varepsilon>0$,
\begin{align}\label{T1}
&\ee^{\qq}\left[\sup_{(t,x)\in[0,T]\times[0,1]^d}|I_1(t,x)|^2\right]\notag\\
\le&  \varepsilon \mathbb E^{\qq}\left[\sup_{(t,x)\in[0,T]\times[0,1]^d}|\sigma(u(t,x))-\sigma(v(t,x))|^2\right]\notag\\
&+ C_{T,2,\eta,\varepsilon} \mathbb E^{\qq}\int_0^T  \sup_{y\in[0,1]^d}|\sigma(u(s,y))-\sigma(v(s,y))|^2 ds\notag\\
\le & \varepsilon L_{\sigma}^2\mathbb E^{\qq}\left[\sup_{(t,x)\in[0,T]\times[0,1]^d}|u(t,x)-v(t,x)|^2\right]\notag\\
&+ C_{T,2,\eta,\varepsilon} L_{\sigma}^2\int_0^T   \mathbb E^{\qq} \sup_{(r,x)\in[0,s]\times[0,1]^d}  | u(r,x) - v(r,x) |^2ds,
\end{align}
where $C_{T,2,\eta,\varepsilon}$ is the constant   in (\ref{101.3}) with $p=2$.
\noindent
By Cauchy-Schwarz's inequality, (\ref{190724.161938}) in Appendix and the Lipschitz continuity of $b$, we obtain that
%\begin{align}\label{T2}
%&\ee^{\qq}\left[\sup_{(t,x)\in[0,T]\times[0,1]^d}|I_2(t,x)|^2\right]\notag\\
% \leq & L_b^2 \left(\sup_{(t,x)\in[0,T]\times[0,1]^d}\int_0^t\int_{[0,1]^d}G_{t-s}^{2}( x,y)dyds\right)\times \left(\int_0^T\int_{[0,1]^d} \ee^{\qq}\left[|u(s,y)-v(s,y)|^2\right]dyds\right)
%  \notag\\
%    \leq & C(L_b,T)\int_0^T\ee^{\qq}\left[\sup_{(r,y)\in[0,s]\times[0,1]^d}|u(r,y)-v(r,y)|^2\right]ds.
%\end{align}
\begin{align}\label{T2}
&\ee^{\qq}\left[\sup_{(t,x)\in[0,T]\times[0,1]^d}|I_2(t,x)|^2\right]\notag\\
 \leq & \left(\sup_{(t,x)\in[0,T]\times[0,1]^d}\int_0^t\int_{[0,1]^d}G_{t-s}( x,y)dyds\right) \nonumber\\
 & \times  \ee^{\qq}\left(\sup_{(t,x)\in[0,T]\times[0,1]^d} \int_0^t\int_{[0,1]^d} G_{t-s}(x,y) \left|b(u(s,y))-b(v(s,y))\right|^2dyds\right)
  \notag\\
    \leq & T L_b^2\ee^{\qq} \left(\sup_{(t,x)\in[0,T]\times[0,1]^d} \int_0^t \sup_{(r,y)\in[0,s]\times[0,1]^d}|u(r,y)-v(r,y)|^2 \int_{[0,1]^d} G_{t-s}(x,y) dy ds \right) \nonumber\\
    \leq & T L_b^2  \int_0^T \ee^{\qq} \sup_{(r,y)\in[0,s]\times[0,1]^d}|u(r,y)-v(r,y)|^2  ds  .
    \end{align}
For the third term, by assumption (C2) and Lemma \ref{Lem t} in Appendix, we have
\begin{align}\label{T3}
&\ee^{\qq}\left[\sup_{(t,x)\in[0,T]\times[0,1]^d}|I_3(t,x)|^2\right]\notag\\
 \leq &  K_{\sigma}^2\left(\sup_{(t,x)\in[0,T]\times[0,1]^d}\int_0^t \|G_{t-s}(x,\cdot)\|_{\HH}^2ds\right)\times \ee^{\qq}\left(\int_0^T\|h(s)\|_{\HH}^2ds\right)
  \notag\\
       \leq & K_{\sigma}^2 C_{G, T,\eta}\ee^{\qq}\int_0^T\|h(s)\|_{\HH}^2ds,
\end{align}
where $ C_{G, T,\eta}$ is the constant given in  \eqref{1907.1931.46}.
Now, for every $T>0$, define
\[
Y(T):=\ee^{\qq}\left[\sup_{(s,x)\in[0,T]\times[0,1]^d}|u(s,x)-v(s,x)|^2 \right].
\]

\noindent Recall that (see e.g. Theorem 3.2 in \cite{MS} or  Proposition 3.2 in \cite{LWZ}, combined with Garsia's lemma)
\begin{align}
   \ee^{\qq}\left[\sup_{(t,x)\in[0,T]\times[0,1]^d}|u(t,x)|^2\right] +
  \ee^{\qq}\left[\sup_{(t,x)\in[0,T]\times[0,1]^d}|v(t,x)|^2\right]< \infty.
\end{align}
Hence $Y(T)<\infty$ for any $T>0$.
Putting \eqref{eq T}-\eqref{T3} together, we have
% (those coefficients should be computed precisely)
\begin{align}
Y(T)\le& 3\varepsilon L_{\sigma}^2 Y(T) + 3(C_{T,2,\eta,\varepsilon} L_{\sigma}^2+TL_{b}^2)\int_0^T Y(s) ds \nonumber\\
&+ 3 K_{\sigma}^2 C_{G, T,\eta} \ee^{\qq}\int_0^T\|h(s)\|_{\HH}^2ds .
\end{align}
Taking $\varepsilon=\varepsilon_0:=\frac{1}{6L_{\sigma}^2}$ and subtracting $\frac{1}{2}Y(T)$ from both sides of the above inequality yield
\begin{align}\label{190724.170724}
	Y(T) \leq 6\left(C_{T,2,\eta,\varepsilon_0} L_{\sigma}^2+TL_{b}^2\right) \int_0^T Y(s) ds+ 6K_{\sigma}^2 C_{G, T,\eta} \ee^{\qq}\int_0^T\|h(s)\|_{\HH}^2ds .
\end{align}
Obviously, (\ref{190724.170724}) still holds if we replace $T$ with any $t\in [0,T]$.
Hence we can use Gronwall's inequality to obtain that
%(see,  e.g., Remark 16 in \cite{Dal}),
\[
Y(T)\le   6K_{\sigma}^2 C_{G, T,\eta} e^{6\left(C_{T,2,\eta,\varepsilon_0} L_{\sigma}^2+TL_{b}^2\right)T  } \ee^{\qq}\int_0^T\|h(s)\|_{\HH}^2ds.
\]
This proves (\ref{eq tar}).
The proof is complete.
  \nprf

\vskip 1.5cm
\section{ Appendix }
\noindent To make reading easier, we present here some results on the kernel $G$ associated with equation (\ref{heat kernel}).

  Let
\[
H_t(x):=\left(\frac{1}{4\pi t}\right)^{\frac d2}\exp\left(-\frac{|x|^2}{4t}\right), \quad x\in \rr^d,\  t>0.
\]
%It is known that the Fourier transform of $H_t$ is
%\[
%\int_{\mathbb{R}^d} e^{-2\pi i x\cdot \xi} H_t(x) dx =e^{-4\pi^2 t |\xi|^2} .
%\]
From Lemma 7 of \cite{van}, we have
\begin{align}
\label{190719.1930.11}	G_t (x,y) \le \ & H_t(x-y), \ \ \forall t>0, x,y\in [0,1]^d.	
\end{align}
Hence, it is easy to see that
\begin{align}
\label{190724.161938}	\int_{[0,1]^d} G_t(x,y) dy < & 1, \\
\label{103.2}  \int_{[0,1]^d} G_t(x,y)^2\,dy < & \sup_{y\in[0,1]^d}G_t(x,y)\times \int_{[0,1]^d} G_t(x,y)\,dy < (4\pi)^{-\frac{d}{2}} t^{-\frac{d}{2}}.
\end{align}

%From the Appendix of M\'{a}rquez-Carreras and Sarr\`{a} \cite{MS}, we have the following estimates.
\begin{lemma}\label{Lem t} Suppose hypothesis $(H_{\eta})$ holds. Then for any $t> 0$ and $x\in [0,1]^{d}$,
\begin{align}
\label{1907.1931.44} \| G_t (x,\cdot)\|_{\HH}^2 \le   K_{\eta} \left(1\vee \left(\frac{\eta}{8\pi^2 }\right)^{\eta}  t^{-\eta}\right),
\end{align}
where  $K_{\eta}$ is the constant in (\ref{H eta1}). Consequently, we have
\begin{align}\label{1907.1931.45} \int_0^T\| G_t (x,\cdot)\|_{\HH}^2dt \le  C_{G, T,  \eta}, \quad \forall x\in [0,1]^{d} ,
\end{align}
 where
 \begin{align}\label{1907.1931.46}
 C_{G, T,  \eta}:=
\left\{
  \begin{array}{ll}
   (1-\eta)^{-1}K_{\eta}\left(\frac{\eta}{8\pi^2}\right)^{\eta} T^{1-\eta},   &  \hbox{ if  $T\le \frac{\eta}{8\pi^2}$};\\
        K_{\eta} T+\frac{\eta^2}{8\pi^2(1-\eta)}, &  \hbox{ if $ T > \frac{\eta}{8\pi^2}$} .
   \end{array}
\right.
\end{align}
\end{lemma}

\begin{proof}
	By (\ref{190719.194202}) and (\ref{190719.1930.11}),  we have for any $t>0$,
\begin{align*}
\|G_t(x,\cdot)\|_{\HH}^2 =&\int_{[0,1]^d}dy_1\int_{[0,1]^d}dy_2 G_{t}(x,y_1)G_{t}(x,y_2)  f(y_1-y_2)\nonumber\\
\le& \int_{\mathbb{R}^d}dy_1\int_{\mathbb{R}^d}dy_2 H_{t}(x-y_1)H_{t}(x-y_2)  f(y_1-y_2)\nonumber\\
=&\int_{\mathbb{R}^d}dy_1\int_{\mathbb{R}^d}dy_2 H_{t}(y_1)H_{t}(y_2)  f(y_1-y_2) \ \nonumber\\
= & \int_{\mathbb{R}^d} e^{-8\pi^2 t|\xi|^2} \lambda(d\xi) \nonumber\\
\le & \sup_{\xi\in\mathbb{R}^d}\left(e^{-8\pi^2 t|\xi|^2}\left(1+|\xi|^2\right)^{\eta}\right) \times \int_{\mathbb{R}^d}  \frac{1}{\left(1+|\xi|^2\right)^{\eta}}\lambda(d\xi) \nonumber\\
\le  &K_{\eta} \left(1\vee \left(\frac{\eta}{8\pi^2 }\right)^{\eta}  t^{-\eta}\right).
\end{align*}
 Based on the above inequality,  it is easily to  obtain  \eqref{1907.1931.45}.
The proof is complete.
\end{proof}

\vskip0.3cm

\noindent{\bf Acknowledgments}: S. Shang is supported by the Fundamental Research Funds for the Central Universities (WK0010000057), Project funded by China Postdoctoral Science Foundation (2019M652174).   R. Wang is supported by National  Natural Science Foundation of China (11871382,  11671076,  11431014).

\end{document}